\documentclass[11pt]{article}

\usepackage{amsmath}
\usepackage{amsthm}
\usepackage{amsfonts}
\usepackage{bbm}

\usepackage{xcolor}

\newcommand{\vf}{\varphi}

\newcommand{\mmp}{\mathbb{P}}

\newcommand{\tp}{\overset{{\rm P}}{\to}}

\newcommand{\me}{\mathbb{E}}

\newcommand{\mr}{\mathbb{R}}
\newcommand{\mn}{\mathbb{N}}

\newcommand{\lin}{\lim_{n\to\infty}}

\newcommand{\lix}{\underset{x\to\infty}{\lim}}

\newcommand{\be}{\begin{equation}}
\newcommand{\ee}{\end{equation}}

\DeclareMathOperator{\1}{\mathbbm{1}}

\newtheorem{thm}{Theorem}[section]
\newtheorem{lemma}[thm]{Lemma}

\newtheorem{assertion}[thm]{Proposition}
\theoremstyle{definition}

\theoremstyle{remark}
\newtheorem{rem}[thm]{Remark}

\begin{document}
\title{Functional limit theorems for the maxima of perturbed random walks and divergent perpetuities in the $M_1$-topology}

\author{Alexander Iksanov\footnote{Faculty of Computer Science and Cybernetics, Taras Shevchenko National University of Kyiv, Kyiv,
Ukraine and Institute of Mathematics, University of Wroc{\l}aw,
50-384 Wroc{\l}aw, Poland; e-mail: iksan@univ.kiev.ua}, Andrey
Pilipenko\footnote{Institute of Mathematics, National Academy of
Sciences of Ukraine, Kyiv, Ukraine; \newline e-mail:
pilipenko.ay@yandex.ua} \ and \ Igor Samoilenko\footnote{Faculty
of Computer Science and Cybernetics, Taras Shevchenko National
University of Kyiv, Kyiv, Ukraine;
\newline e-mail: isamoil@i.ua}}
\maketitle
\begin{abstract}
\noindent Let $(\xi_1,\eta_1)$, $(\xi_2,\eta_2),\ldots$ be a
sequence of i.i.d.\ two-dimensional random vectors. In the earlier article Iksanov and Pilipenko (2014) weak convergence in the $J_1$-topology on the Skorokhod space of $n^{-1/2}\underset{0\leq k\leq \cdot}{\max}\,(\xi_1+\ldots+\xi_k+\eta_{k+1})$ was proved under the assumption that contributions of $\underset{0\leq k\leq
n}{\max}\,(\xi_1+\ldots+\xi_k)$ and $\underset{1\leq k\leq
n}{\max}\,\eta_k$ to the limit are comparable and that $n^{-1/2}(\xi_1+\ldots+\xi_{[n\cdot]})$ is attracted to a Brownian motion. In the present paper, we continue this line of research and investigate a more complicated situation when $\xi_1+\ldots+\xi_{[n\cdot]}$, properly normalized without centering, is attracted to a centered stable L\'{e}vy process, a process with jumps. As a consequence, weak convergence normally holds in the $M_1$-topology. We also provide sufficient conditions for the $J_1$-convergence. For completeness, less interesting situations are discussed when one of the sequences $\underset{0\leq k\leq
n}{\max}\,(\xi_1+\ldots+\xi_k)$ and $\underset{1\leq k\leq
n}{\max}\,\eta_k$ dominates the other. An application of our main results to divergent perpetuities with positive entries is given.
\end{abstract}

\noindent Key words: functional limit theorem; $J_1$-topology;
$M_1$-topology; perpetuity; perturbed random walk

\section{Introduction and results}

Let $(\xi_k, \eta_k)_{k\in\mn}$ be a sequence of i.i.d.\
two-dimensional random vectors with generic copy $(\xi,\eta)$. No
condition is imposed on the dependence structure between $\xi$ and
$\eta$. Set $\mn_0:=\mn\cup\{0\}$. Further, let
$(S_n)_{n\in\mn_0}$ be the zero-delayed ordinary random walk with
increments $\xi_n$ for $n\in\mn$, i.e., $S_0 = 0$ and $S_n =
\xi_1+\ldots+\xi_n$, $n \in \mn$. Then define its perturbed
variant $(T_n)_{n\in\mn}$, that we call {\it perturbed random
walk}, by
\begin{equation}    \label{eq:PRW}
T_n := S_{n-1} + \eta_n,    \quad   n \in \mn.
\end{equation}
Recently it has become a rather popular object of research, see
the recent book \cite{Iksanov:2017} for a survey and
\cite{Alsmeyer+Iksanov+Meiners:2015, Araman+Glynn:2006,
Hao+Tang+Wei:2009, Iksanov+Pilipenko:2014, Meerschaert+Stoev:2009,
Palmowski+Zwart:2007, Palmowski+Zwart:2010,
Pancheva+Jordanova:2004, Pancheva+Mitov+Mitov:2009, Robert:2005,
Wang:2014}. It is worth noting that sometimes in the literature
the term `perturbed random walk' was used to denote random
sequences other than those defined in \eqref{eq:PRW}. See, for
instance, \cite{Chen+Sun:2014+, Davis:1996,
Iksanov+Pilipenko:2015, Lai+Siegmund:1977, Lai+Siegmund:1979, Pilipenko+Prihodko:2015,
Woodroofe:1982} and Section 6 in \cite{Gut:2009}.

Denote by $D:=D[0,\infty)$ the Skorokhod space of real-valued
right-continuous functions which are defined on $[0,\infty)$ and
have finite limits from the left at each positive point.
Throughout the paper we assume that $D$ is equipped with either
the $J_1$-topology or the $M_1$-topology. We refer to
\cite{Billingsley:1968, Jacod+Shiryaev:2003} and \cite{Whitt:2002}
for comprehensive accounts of the $J_1$- and the $M_1$-topologies,
respectively. We write $\mathcal{M}_p(A)$ for the set of Radon
point measures on a `nice' space $A$. The $\mathcal{M}_p(A)$ is
endowed with vague convergence. More information on these can be
found in \cite{Resnick:2007}. Throughout the paper $\overset{{\rm
J_1}}{\Rightarrow}$ and $\overset{{\rm M_1}}{\Rightarrow}$ will
mean weak convergence on the Skorokhod space $D$ when endowed with
the $J_1$-topology and the $M_1$-topology, respectively. The
notation $\Rightarrow$ without superscript is normally followed by
a specification of the topology and the space involved. Finally,
we write $\overset{{\rm v}}{\to}$ and $\overset{{\rm P}}{\to}$ to
denote vague convergence and convergence in probability,
respectively.

In the present paper we are interested in weak convergence on $D$
of $\max_{0\leq k\leq [n\cdot]}\,(S_k+\eta_{k+1})$, properly
normalized without centering, as $n\to\infty$. It should not come
as a surprise that the maxima exhibit three types of different
behaviors depending on the asymptotic interplay of
$A_n:=\underset{0\leq k\leq n}{\max}\,S_k$ and
$B_n:=\underset{1\leq k\leq n+1}{\max}\,\eta_k$, namely, on
whether (I) $A_n$ dominates $B_n$; (II) $A_n$ is dominated by
$B_n$; (III) $A_n$ and $B_n$ are comparable.

Relying essentially upon findings in \cite{Iksanov+Pilipenko:2014}
three functional limit theorems for the maxima of perturbed random
walks, properly rescaled without centering, were proved in
\cite{Iksanov:2017} under the assumption that $\me\xi^2<\infty$.
Throughout the remainder of the paragraph we assume that the most
interesting alternative (III) prevails. The situation treated in
\cite{Iksanov+Pilipenko:2014} was relatively simple because the
limit process for $S_{[n\cdot]}/n^{1/2}$ was a Brownian motion, a
process with continuous paths. As a consequence, the convergence
took place in the $J_1$-topology on $D$, and, more surprisingly,
the contributions of $(S_k)$ and $(\eta_j)$ turned out
asymptotically independent, despite the possible strong dependence
of $\xi$ and $\eta$. In the present paper we treat a more delicate
case where the distribution of $\xi$ belongs to the domain of
attraction of an $\alpha$-stable distribution, $\alpha\in (0,2)$,
so that the limit process for $S_{[n\cdot]}$, properly normalized,
is an $\alpha$-stable L\'{e}vy process. We shall show that the
presence of jumps in the latter process destroys dramatically an
idyllic picture pertaining to the Brownian motion scenario: the
convergence typically holds in the weaker $M_1$-topology on $D$,
and the aforementioned asymptotic independence only occurs in some
exceptional cases where $\xi$ and $\eta$ are themselves
asymptotically independent in an appropriate sense.

Throughout the paper we assume that, as $x\to\infty$,
\begin{equation}\label{cond10}
\mmp\{|\xi|>x\}\sim x^{-\alpha}\ell(x)
\end{equation}
and that \begin{equation}\label{cond100} \mmp\{\xi>x\}\sim
c_1\mmp\{|\xi|>x\},\quad \mmp\{-\xi>x\}\sim c_2\mmp\{|\xi|>x\}
\end{equation}
for some $\alpha\in (0,2)$, some $\ell$ slowly varying at
$\infty$, some nonnegative $c_1$ and $c_2$ summing up to one. The
assumptions mean that the distribution of $\xi$ belongs to the
domain of attraction of an $\alpha$-stable distribution. To ensure
weak convergence of $S_{[n\cdot]}$ without centering we assume
that $\me\xi=0$ when $\alpha\in (1,2)$ and that the distribution
of $\xi$ is symmetric when $\alpha=1$. Then the classical
Skorokhod theorem (Theorem 2.7 in \cite{Skorokhod:1957}) tells us
that
\begin{equation}\label{skor}
\frac{S_{[n\cdot]}}{a(n)}~\overset{{\rm J_1}}{\Rightarrow}~
\mathcal{S}_\alpha(\cdot),\quad n\to\infty,
\end{equation}
where $a(x)$ is a positive function satisfying $\lix
x\mmp\{|\xi|>a(x)\}=1$ and
$\mathcal{S}_\alpha:=(\mathcal{S}_\alpha(t))_{t\geq 0}$ is an
$\alpha$-stable L\'{e}vy process with the characteristic function
$$\me\exp({\rm i}z\mathcal{S}_\alpha(1))=\exp\big(|z|^\alpha (\Gamma(2-\alpha)/(\alpha-1))(\cos(\pi\alpha/2)-{\rm i}(c_1-c_2)\sin(\pi\alpha/2)
{\rm sign} z)\big),\quad z\in\mr$$ when $\alpha\in (0,2)$,
$\alpha\neq 1$, here, $\Gamma(\cdot)$ denotes the Euler gamma
function, and
$$\me\exp({\rm i}z\mathcal{S}_\alpha(1))=\exp(-2^{-1}\pi|z|),\quad
z\in\mr$$ when $\alpha=1$.

Put $E:=[-\infty, +\infty]\times [0,\infty]\backslash\{(0,0)\}$.
For a Radon measure $\rho$ on $E$, let $N^{(\rho)}:=\sum_k
\varepsilon_{(\theta_k,\,i_k,\,j_k)}$ be a Poisson random measure
on $[0,\infty)\times E$ with mean measure $\mathbb{LEB}\times
\rho$, where $\varepsilon_{(t,\,x,\, y)}$ is the probability
measure concentrated at $(t,x,y)\subset [0,\infty)\times E$,
$\mathbb{LEB}$ is the Lebesgue measure on $[0,\infty)$.

Theorem \ref{main} which is the main result of the present paper
treats the most complicated situation (III) when the contributions
of $\underset{0\leq k\leq n}{\max}\,S_k$ and $\underset{1\leq
k\leq n+1}{\max}\,\eta_k$ to the asymptotic behavior of
$\underset{0\leq k\leq n}{\max}\,(S_k+\eta_{k+1})$ are comparable.
At this point it is worth stressing that $\xi$ and $\eta$ are
assumed arbitrarily dependent which makes the analysis nontrivial.
We stipulate hereafter that the supremum over the empty set equals
zero.
\begin{thm}\label{main}
Suppose that conditions \eqref{cond10} and \eqref{cond100} hold,
that $\mmp\{\eta>x\}\sim c\mmp\{|\xi|>x\}$ as $x\to\infty$, for
some $c>0$, and that
\begin{equation}\label{princ}
x\mmp\bigg\{\frac{(\xi,\eta^+)}{a(x)}\in
\cdot\bigg\}~\overset{{\rm v}}{\to}~\nu,\quad x\to\infty
\end{equation}
on $\mathcal{M}_p(E)$. Then
\begin{equation}\label{limit_perturbed01}
\frac{\underset{0\leq k\leq [n
\cdot]}{\max}\,(S_k+\eta_{k+1})}{a(n)}\quad\overset{{\rm
M_1}}{\Rightarrow}\quad \sup_{0\leq
s\leq\cdot}\,\mathcal{S}^\ast_\alpha(s)\vee \sup_{\theta_k\leq
\cdot} \,(\mathcal{S}^\ast_\alpha(\theta_k-)+j_k),\quad
n\to\infty,
\end{equation}
where $(\theta_k, i_k, j_k)$ are the atoms of a Poisson random
measure $N^{(\nu)}$ and $\mathcal{S}_\alpha^\ast$ is a copy of
$\mathcal{S}_\alpha$ whose L\'{e}vy-It\^{o} representation is
built upon the Poisson random measure
$\sum_k\varepsilon_{(\theta_k, i_k)}$.

Under the additional assumption
\begin{equation}\label{degen}
\nu\{(x,y): 0<y<x\}=0,
\end{equation}
the convergence in \eqref{limit_perturbed01} holds in the
$J_1$-topology on $D$.
\end{thm}

We proceed with a number of remarks.
\begin{rem}
It is perhaps worth stating explicitly that $N^{(\nu)}(\cdot\times
[-\infty,+\infty]\times \cdot)=\sum_k
\varepsilon_{(\theta_k,\,j_k)}$ is a Poisson random measure on
$[0,\infty)\times (0,\infty]$ with mean measure
$\mathbb{LEB}\times \mu_c$, where $\mu_c$ is a measure on
$(0,\infty]$ defined by
$$\mu_c\big((x,\infty]\big)=cx^{-\alpha},\quad x>0.$$ Analogously,
$N^{(\nu)}(\cdot\times\cdot\times [0,\infty])=\sum_k
\varepsilon_{(\theta_k,\,i_k)}$ is a Poisson random measure on
$[0,\infty)\times ([-\infty,+\infty]\backslash\{0\})$ with mean
measure $\mathbb{LEB}\times \nu^\ast$, where $\nu^\ast$ is the
L\'{e}vy measure of $\mathcal{S}_\alpha$ given by
$$\nu^\ast((x,\infty])=c_1 x^{-\alpha},\quad \nu^\ast((-\infty, -x])=c_2x^{-\alpha},\quad x>0.$$
\end{rem}
\begin{rem}
Condition \eqref{degen} obviously holds if the measure $\nu$ is
concentrated on the axes. This is the case whenever $\xi$ and
$\eta$ are independent and also in many cases when these are
dependent. For instance, take $\xi=|\log W|$ and
$\eta=|\log(1-W)|$ satisfying \eqref{princ} for a random variable $W$ taking values in
$(0,1)$ (details can be found in the proof of Theorem 1.1 in
\cite{Iksanov+Marynych+Vatutin:2015}).

Suppose now that $\eta=r\xi$ for some $r>0$. Then the restriction
of $\nu$ to the first quadrant concentrates on the line $y=rx$.
Hence, condition \eqref{degen} holds if, and only if, $r\geq 1$.

Let $\rho$ be a positive random variable such that $\mmp\{\rho>x\}\sim
x^{-\alpha}$ as $x\to\infty$ and $\zeta$ a random variable which is
independent of $\rho$ and takes values in $[-\pi,\pi)$. Setting $\xi=\rho
\cos\zeta$ and $\eta=\rho\sin\zeta$ we obtain
$$\mmp\{\xi>x\}~\sim~ (\me (\cos\zeta)^\alpha\1_{\{|\zeta|<\pi/2\}}) x^{-\alpha},~ \mmp\{-\xi>x\}~\sim~(\me
|\cos\zeta|^\alpha\1_{\{|\zeta|>\pi/2\}}) x^{-\alpha}$$ as
$x\to\infty$ by the Lebesgue dominated convergence theorem.
Furthermore,
$$
x\mmp\bigg\{\frac{(\xi,\eta^+)}{a(x)}\in
\cdot\bigg\}~\overset{{\rm v}}{\to}~\nu,\quad x\to\infty,
$$
where $a(x)=(\me |\cos\zeta|^\alpha)^{1/\alpha}x^{1/\alpha}$ and $\nu$ is
the image of the measure $$\frac{\alpha\me
(\cos\zeta)^\alpha\1_{\{\zeta\in (-\pi/2,\pi/2)\}}}{\me
|\cos\zeta|^\alpha}\1_{(r,\vf)\in(0,\infty)\times[0,\pi)}
r^{-\alpha-1}{\rm d}r \mmp\{\zeta\in {\rm d}\varphi\}$$
under the mapping $(r,\vf)\to (r\cos\vf, r\sin\vf)$.
Condition \eqref{degen} is equivalent to $\mmp\{\zeta\in (0,\pi/4)\}=0$.
\end{rem}

\begin{rem}
Weak convergence of nondecreasing processes in the $M_1$-topology
is not as strong as it might appear. Actually, it is equivalent to
weak convergence of finite-dimensional distributions just because
a sequence of nondecreasing processes is always tight on $D$
equipped with the $M_1$-topology. This follows from the fact that
the $M_1$-oscillation $$\omega_\delta(f):=\sup_{t_1\leq t\leq t_2,
0\leq t_2-t_1\leq \delta}\, M(f(t_1), f(t), f(t_2))$$ of a
nondecreasing function $f$ equals zero, where $M(x_1,x_2,x_3):=0$,
if $x_2\in [x_1, x_3]$, and $:=\min(|x_2-x_1|, |x_3-x_2|)$,
otherwise.
\end{rem}
\begin{rem}
From a look at Theorem \ref{main10} underlying the proof of Theorem \ref{main} it should be clear that a counterpart of Theorem \ref{main} also holds when replacing the input vectors $(\xi_k, \eta_k)_{k\in\mn}$ with arrays $\big(\xi_k^{(n)}, \eta_k^{(n)}\big)_{k\in\mn}$ for each $n\in\mn$. We however refrain from formulating such a generalization, for we are not aware of any potential applications of such a result.
\end{rem}

Propositions \ref{main1} and \ref{main2} given next are concerned
with the simpler situations (I) and (II), respectively.
\begin{assertion}\label{main1}
Suppose that conditions \eqref{cond10} and \eqref{cond100} hold
and that
\begin{equation}\label{cond1}
\lix \frac{\mmp\{\eta>x\}}{\mmp\{|\xi|>x\}}=0.
\end{equation}
Then
\begin{equation}\label{limit_perturbed0}
\frac{\underset{0\leq k\leq [n
\cdot]}{\max}\,(S_k+\eta_{k+1})}{a(n)}\quad\overset{{\rm
J_1}}{\Rightarrow}\quad \underset{0\leq s\leq
\cdot}{\sup}\,\mathcal{S}_\alpha(s),\quad n\to\infty.
\end{equation}
\end{assertion}
\begin{rem}
When $\alpha\in (0,1)$ and $c_1=0$, the right-hand side in
\eqref{limit_perturbed0} is the zero function because
$\mathcal{S}_\alpha$ is then the negative of an $\alpha$-stable
subordinator (recall that a subordinator is a nondecreasing
L\'{e}vy process). In this setting there are two possibilities:
either $\sup_{k\geq 0}(S_k+\eta_{k+1})<\infty$ a.s.\ or
$\sup_{k\geq 0}(S_k+\eta_{k+1})=\infty$ a.s. Plainly, if the first
alternative prevails, much more than \eqref{limit_perturbed0} can
be said, namely, $\underset{0\leq k\leq [n
\cdot]}{\max}\,(S_k+\eta_{k+1})/r_n$ converges to the zero
function in the $J_1$-topology on $D$ for any positive sequence
$(r_n)$ diverging to $\infty$.

Now we intend to give examples showing that either of
possibilities can hold. By Theorem 2.1 in
\cite{Alsmeyer+Iksanov+Meiners:2015}, the supremum of
$S_k+\eta_{k+1}$ is finite a.s.\ if, and only if,
\begin{equation}\label{check}
\int_{(0,\infty)}\frac{x}{\int_0^x \mmp\{-\xi>y\}{\rm d}y}{\rm
d}\mmp\{\xi\leq y\}<\infty\quad \text{and}\quad
\int_{(0,\infty)}\frac{x}{\int_0^x \mmp\{-\xi>y\}{\rm d}y}{\rm
d}\mmp\{\eta\leq y\}<\infty
\end{equation}
If $\mmp\{-\xi>x\}\sim x^{-\alpha}$, $\mmp\{\xi>x\}\sim
x^{-\alpha}(\log x)^{-2}$ as $x\to\infty$ and $\me
(\eta^+)^\alpha<\infty$, then both inequalities in \eqref{check}
hold, whereas if $\mmp\{-\xi>x\}\sim x^{-\alpha}$,
$\mmp\{\xi>x\}\sim x^{-\alpha}(\log x)^{-1}$ as $x\to\infty$, then
the first integral in \eqref{check} diverges.
\end{rem}

\begin{assertion}\label{main2}
Suppose that conditions \eqref{cond10} and \eqref{cond100} hold,
that
\begin{equation}\label{cond12}
\lix \frac{\mmp\{\eta>x\}}{\mmp\{|\xi|>x\}}=\infty
\end{equation}
and that $\mmp\{\eta>x\}$ is regularly varying at $\infty$ of
index $-\beta$ (necessarily $\beta\in (0,\alpha]$). Let $b(x)$ be
a positive function which satisfies $\lix x\mmp\{\eta>b(x)\}=1$.
Then
\begin{equation}\label{limit_perturbed1} \frac{\underset{0\leq
k\leq [n \cdot]}{\max}\,(S_k+\eta_{k+1})}{b(n)}~\overset{{\rm
J_1}}{\Rightarrow}~ \sup_{\theta_k \leq \cdot}\,j_k,\quad
n\to\infty,
\end{equation}
where $(\theta_k, j_k)$ are the atoms of a Poisson random measure
on $[0,\infty)\times (0,\infty]$ with mean measure
$\mathbb{LEB}\times \mu$, where $\mu$ is a measure on $(0,\infty]$
defined by
$$\mu\big((x,\infty]\big)=x^{-\beta},\quad x>0.$$
\end{assertion}

Whenever the random series $\sum_{k\geq 0}e^{T_{k+1}}$ converges
a.s., its sum is called {\it perpetuity} due to its occurrence in
the realm of insurance and finance as a sum of discounted payment
streams. When the random series diverges, it is natural to
investigate weak convergence on $D$ of its partial sums, properly
rescaled, as the number of summands becomes large. Some results of
this flavor can be found in \cite{Buraczewski+Iksanov:2015} and
\cite{Iksanov:2017} (in these works many references to earlier
one-dimensional results can be found). Here, we prove functional
limit theorems in the situations that remained untouched.
\begin{thm}\label{main3}
In the settings of Theorem \ref{main} and Propositions \ref{main1}
and \ref{main2} functional limit theorems hold with $\log
\bigg(\sum_{k=0}^{[n\cdot]}e^{T_{k+1}}\bigg)$ replacing $\max_{0\leq
k\leq [n\cdot]}\,T_{k+1}$.
\end{thm}

\section{Proof of Theorem \ref{main}}

For each $n\in\mn$, let $\big(x^{(n)}_i, y^{(n)}_i\big)_{i\in\mn}$ be a
sequence of $\mr^2$-valued vectors. Put $S_0^{(n)}:=0$,
$$
S^{(n)}_k:=\sum_{i=1}^k x^{(n)}_i,\quad k\in\mn, \quad T^{(n)}_k:=S_k^{(n)} +
y^{(n)}_{k+1}, \quad k\in\mn_0
$$
and then define the piecewise constant functions
$$
f_n(t):=\sum_{k\geq 0} S^{(n)}_k
\1_{[\frac{k}{n},\frac{k+1}{n})}(t),\quad g_n(t):=\max_{0\leq
k\leq [nt]} T^{(n)}_k,\quad t\geq 0
$$
where $\1_A(x)=1$ if $x\in A$ and $=0$, otherwise.

To proceed, we have to recall the notation
$E=[-\infty,+\infty]\times [0,\infty]\backslash\{(0,0)\}$. The proof
of Theorem \ref{main} is essentially based on the following
deterministic result along with the continuous mapping theorem.
\begin{thm}\label{main10}
Let $f_0\in D$ and $\nu_0=\sum_k\varepsilon_{(t_k, x_k,y_k)}$ be a
Radon measure on $[0,\infty)\times E$ satisfying
$\nu_0(\{0\}\times E)=0$ and $t_k\neq t_j$ for $k\neq j$. Suppose
that
\begin{equation}\label{eq1}
\lin f_n= f_0
\end{equation}
in the $J_1$-topology on $D$ and that
\begin{equation}\label{eq1_1}
\nu_n:=\sum_{i\geq
1}\varepsilon_{(i/n,\,x_i^{(n)},\,y_i^{(n)})}\1_{\{y_i^{(n)}>0\}}\overset{{\rm
v}}{\to}\nu_0,\quad n\to\infty
\end{equation}
on $\mathcal{M}_p([0,\infty)\times E)$. Then $$\lin g_n=g_0:=\sup_{0\leq
s\leq \cdot}\,f_0(s)\vee \sup_{t_k\leq \cdot}\,(f_0(t_k-)+y_k)$$
in the $M_1$-topology on $D$. This convergence holds in the
$J_1$-topology on $D$ under the additional assumption
\begin{equation}\label{eq1dop}
\nu_0([0,\infty)\times\{(x,y): 0<y<x\})=0.
\end{equation}
\end{thm}

\begin{rem}\label{re2}
Suppose that $f_0$ is continuous and that the set of points
$(t_k,y_k)$ with $y_k>0$ is dense in $[0,\infty)$. Then
$$
g_0(\cdot)=\sup_{t_k\leq \cdot}\,(f_0(t_k)+ y_k).
$$
Furthermore, condition \eqref{eq1dop} holds automatically, and
condition \eqref{eq1_1} is equivalent to
\begin{equation*}
\sum_{i\geq
1}\varepsilon_{(i/n,\,y_i^{(n)})}\1_{\{y_i^{(n)}>0\}}\overset{{\rm
v}}{\to}\sum_k \varepsilon_{(t_k, y_k)}\quad n\to\infty.
\end{equation*}
on $\mathcal{M}_p([0,\infty)\times (0,\infty])$.
This is
the setting of Theorem 1.3 in \cite{Iksanov+Pilipenko:2014}.
\end{rem}
\begin{rem}
Here, we discuss the necessity of condition \eqref{eq1dop} for the
$J_1$-convergence. Suppose that in the setting of Theorem
\ref{main10} there exists $k\in\mn$ such that $0<y_k<x_k$, so that
condition \eqref{eq1dop} does not hold. Now we give an example in
which the $J_1$-convergence in Theorem \ref{main10} fails to hold.
With $x^{(n)}_i=y^{(n)}_i=0$ for $i\neq[n/2]$, $x^{(n)}_{[n/2]}=2$
and $y^{(n)}_{[n/2]}=1$ we have $f_n(t)=2\1_{[[n/2]/n,\infty)}(t)$
and
$$
g_n(t)=\begin{cases}
0, \ t<\frac{[n/2]-1}n,\\
1, \frac{[n/2]-1}n\leq t< \frac{[n/2]}n,\\
2, \frac{[n/2]}n\leq t.
\end{cases}
$$
Plainly, condition \eqref{eq1_1} holds with
$\nu_0=\varepsilon_{(1/2, 2,1)}$. Setting
$f_0(t)=g(t):=2\1_{[1/2,\,\infty)}(t)$ we conclude that $\lin
f_n=f_0$ in the $J_1$-topology and $\lin g_n=g$ in the
$M_1$-topology. On the other hand, $g_n$ has a jump of magnitude
$1$ at point $[n/2]/n$. Furthermore, this magnitude does not
converge to $2$, the size of the limit jump at point $1/2$. This
precludes the $J_1$-convergence.
\end{rem}

Lemma \ref{lem1} given next collects known criteria for the
convergence of nondecreasing functions in the $J_1$- and
$M_1$-topologies. For part (a), see Corollary 12.5.1 and Lemma
12.5.1 in \cite{Whitt:2002}. While one implication of part (b) is
standard, the other follows from Theorem 2.15 on p.~342 and Lemma
2.22 on p.~343 in \cite{Jacod+Shiryaev:2003}.

\begin{lemma}\label{lem1}
Let $(h_n)_{n\in\mn_0}$ be a sequence of nondecreasing functions
in $D$.

\noindent (a) $\lin h_n=h_0$ in the $M_1$-topology on $D$ if, and
only if, $h_n(t)$ converges to $h_0(t)$ for each $t$ in a dense
subset of continuity points of $h_0$ including zero.

\noindent (b) $\lin h_n=h_0$ in the $J_1$-topology on $D$ if, and
only if, $\lin h_n=h_0$ in the $M_1$-topology on $D$ and for any
discontinuity point $s$ of $h_0$ there exists a sequence
$(s_n)_{n\in\mn}$ such that $\lin s_n=s$,
$$ \lin h_n (s_n-)=h(s-)\quad \text{and}\quad \lin
h_n(s_n)=h(s).$$
\end{lemma}
\begin{proof}[Proof of Theorem \ref{main10}]
We start by showing that $g_0\in D$. Since $g_0$ is nondecreasing,
it has finite limits from the left on $(0,\infty)$. Using
right-continuity of $f_0$ we obtain
\begin{eqnarray*}
g_0(t)&\leq& \lim_{\delta\to 0+}\, g_0(t+\delta)= \lim_{\delta\to
0+}\,\big(\sup_{0\leq s<t+\delta}\,f_0(s)\vee \sup_{t_k\leq
t+\delta}\,(f_0(t_k-)+y_k)\big)\\&=& \sup_{0\leq s\leq
t}\,f_0(s)\vee {\underset{\delta\to 0+}{\lim}} \sup_{t_k\leq
t+\delta}\,(f_0(t_k-)+y_k)\\&\leq& \sup_{0\leq s\leq
t}\,f_0(s)\vee \sup_{t_k\leq t}\,(f_0(t_k-)+y_k)\vee
 {\underset{\delta\to 0+}{\lim}} \sup_{t<t_k\leq
 t+\delta}\,(f_0(t_k-)+y_k)\\&\leq&
\sup_{0\leq s\leq t}\,f_0(s)\vee \sup_{t_k\leq
t}\,(f_0(t_k-)+y_k)\vee
 \big({\underset{\delta\to 0+}{\lim}} \sup_{t<t_k\leq t+\delta}\,f_0(t_k-)
 + \, {\underset{\delta\to 0+}{\lim}}\sup_{t<t_k\leq
 t+\delta}y_k\big)\\&=&
\sup_{0\leq s\leq t}\,f_0(s)\vee \sup_{t_k\leq
t}\,(f_0(t_k-)+y_k)\vee (f_0(t)+0)=g_0(t)
\end{eqnarray*}
for any $t>0$ which proves right-continuity of $g_0$.

\noindent {\sc Proof of the $M_1$-convergence}. Since $f_0,g_0\in D$, they have at most countably many discontinuities.
Hence, the set
$$K:=\{T\geq 0: \nu_0(\{T\}\times E)=0;\quad T~\text{is a continuity point of}~ f_0~\text{and
continuity point of}~ g_0\}$$ is dense in $[0,\infty)$. 
Since $g_n$ is nondecreasing for each $n\in\mn$, according to
Lemma \ref{lem1} (a), it suffices to prove that
\begin{equation}\label{conv}
\lin g_n(T)=g_0(T)
\end{equation}
for all $T\in K$. Observe that $g_0(0)=0$ as a consequence of
$f_n(0)=f_0(0)=0$ and $\nu_0(\{0\}\times E)=0$. The last condition
ensures that $g_n(0)=y_1^{(n)}$ converges to zero as $n\to\infty$.
This proves that relation \eqref{conv} holds for $T=0$. Thus, in
what follows we assume that $T\in K$ and $T>0$.

Fix any such a $T$. There exists a sequence
$(\varepsilon_k)_{k\in\mn}$ that vanishes as $k\to\infty$ and such
that its generic element denoted by $\varepsilon$ is a continuity
point of the nonincreasing function $$x\mapsto \nu_0([0,T]\times
[-\infty,+\infty]\times (x,\infty]),$$ so that $\nu_0([0,T]\times
[-\infty,+\infty]\times \{\varepsilon\})=0$. Put
$E_\varepsilon:=[-\infty, +\infty]\times (\varepsilon,\infty]$.
Condition \eqref{eq1_1} implies that $\nu_0([0,T]\times
E_\varepsilon)=\nu_n([0,T]\times E_\varepsilon)=m$ for large
enough $n$ and some $m\in\mn_0$, where the finiteness of $m$ is
secured by the fact that $\nu_0$ is a Radon measure. The case
$m=0$ is trivial. Hence, in what follows we assume that $m\in\mn$.
Denote by $(\bar{t}_i, \bar{x}_i, \bar{y}_i)_{1\leq i\leq m}$ an
enumeration of the points of $\nu_0$ in $[0,T]\times
E_\varepsilon$ with
\begin{equation}\label{ineq}
\bar{t}_1<\bar{t}_2<\ldots<\bar{t}_m
\end{equation}
and by $\big(\bar{t}_i^{(n)}, \bar{x}_i^{(n)},
\bar{y}_i^{(n)}\big)_{1\leq i\leq m}$ the analogous enumeration of
the points $\nu_n$ in $[0,T]\times E_\varepsilon$. Note that
$\bar{t}_1>0$ in view of the assumption $\nu_0(\{0\}\times E)=0$,
whereas the assumption $t_k\neq t_j$ for $k\neq j$ ensures that
the inequalities in \eqref{ineq} are strict. Then
\begin{equation}\label{con1}
\lin \sum_{i=1}^m(| \bar{t}^{(n)}_i- \bar{t}_i|+|\bar{x}^{(n)}_i-
\bar{x}_i|+|\bar{y}^{(n)}_i-\bar{y}_i|)=0.
\end{equation}

Later on we shall need the following relation
\begin{equation}\label{con2}
f_n(\bar{t}_i^{(n)}-1/n)=f_n(\bar{t}_i^{(n)}-)\to
f_0(\bar{t}_i-),\quad n\to\infty
\end{equation}
for $i=1,\ldots, m$. To prove it, fix any $i=1,\ldots, m$ and
assume that $\bar{t}_i$ is a discontinuity point of $f_0$. Then
condition \eqref{eq1} in combination with
$f_n(\bar{t}_i^{(n)})-f_n(\bar{t}_i^{(n)}-)=\bar{x}_i^{(n)}\to
\bar{x}_i\neq 0$ as $n\to\infty$ entails
$\bar{x}_i=f_0(\bar{t}_i)-f_0(\bar{t}_i-)$ and \eqref{con2} (see
the proof of Proposition 2.1 on p.~337 in
\cite{Jacod+Shiryaev:2003}). If $\bar{t}_i$ is a continuity point
of $f_0$, \eqref{con2} holds trivially. Arguing similarly, we also
obtain
\begin{equation}\label{con3}
\lim_{n\to\infty}\,\max_{t\in
[0,t_i^{(n)}-2/n]}\,f_n(t)=\sup_{t\in [0,\bar{t}_i)}\,f_0(t)
\end{equation}
for $i=1,\ldots, m$.

We first work with functions $g_{n,\varepsilon}$ and
$g_{0,\varepsilon}$ which are counterparts of $g_n$ and $g_0$
based on the restrictions of $\nu_n$ and $\nu_0$ to $[0,T]\times
E_\varepsilon$. Put $$A_{n,\,T}:=\{j\in\mn_0: 0\leq j\leq [nT],\
(j+1)/n\neq \bar{t}_i^{(n)} \ \text{for} \ i=1,\ldots, m \}.$$ Now
we are ready to write a basic decomposition
\begin{eqnarray}
g_{n,\varepsilon}(T)&:=&\max_{0\leq i\leq
[nT]}\big(x^{(n)}_1+\ldots+x^{(n)}_i+y^{(n)}_{i+1}\1_{\{y^{(n)}_{i+1}>\varepsilon\}}\big)\notag\\&=&
\max_{i\in A_{n,T}}\,f_n(i/n)\vee \max_{1\leq k\leq
m}\left(f_n(\bar{t}_k^{(n)}-1/n)+\bar{y}^{(n)}_k
\right)\notag\\&=& \max_{0\leq i\leq [nT]}\,f_n(i/n)\vee
\max_{1\leq k\leq m}\left(f_n(\bar{t}_k^{(n)}-1/n)+\bar{y}^{(n)}_k
\right)\notag\\&=&\max_{t\in [0,T]}\,f_n(t)\vee \max_{1\leq k\leq
m}\left(f_n(\bar{t}_k^{(n)}-1/n)+\bar{y}_k^{(n)}
\right),\label{inter2}
\end{eqnarray}
the third equality following from the fact that, for integer $i\in
[0, [nT]]$ such that $i/n=\bar{t}_k^{(n)}-1/n$ for some
$k=1,\ldots, m$, we have $f_n(i/n)<\max_{1\leq k\leq
m}\left(f_n(\bar{t}_k^{(n)}-1/n)+\bar{y}^{(n)}_k \right)$ because
all the $\bar{y}^{(n)}_k$ are positive.

It is convenient to state the following known result as a lemma,
for it will be used twice in the subsequent proof.
\begin{lemma}\label{imp}
Let $s_0$ be a continuity point of $f_0$ and $(s_n)_{n\in\mn}$ a
sequence of positive numbers converging to $s_0$ as $n\to\infty$.
Then $\lin \sup_{t\in [0,s_n]} f_n(t)=\max_{t\in
[0,s_0]}\,f_0(t)$.
\end{lemma}
\begin{proof}
We first observe that $\sup_{t\in [0,s_0]}\,f_0(t)=\max_{t\in
[0,s_0]}\,f_0(t)$ because $s_0$ is a continuity point of $f_0$
(hence, of the supremum). It is well-known (and easily checked)
that \eqref{eq1} entails
\begin{equation}\label{supr}
\lin \sup_{0\leq t\leq \cdot}\,f_n(t)=\sup_{0\leq t\leq
\cdot}\,f_0(t)
\end{equation}
in the $J_1$-topology on $D$. In particular, $\lin \sup_{t\in
[0,s_n]}\,f_n(t)=\max_{t\in [0,s_0]}\,f_0(t)$.
\end{proof}

Recalling that $T$ is a continuity point of $f_0$ and using Lemma
\ref{imp} with $s_n=T$ for all $n\in\mn_0$ we infer
$$
\lin \max_{t\in [0,T]}\,f_n(t) = \sup_{t\in [0,T]}\,f_0(t)
$$
and thereupon
\begin{equation}\label{aux}
\lin g_{n,\varepsilon}(T)=\sup_{s\in[0,T]}f_0(s)\vee
\sup_{\bar{t}_k\leq T}\left(f_0(\bar{t}_{k}-)+\bar y_{k}
\right):=g_{0,\varepsilon}(T)
\end{equation}
having utilized \eqref{con1} and \eqref{con2} for the second
supremum.

Further, we claim that
\begin{equation}\label{121}
\sup_{t\geq 0}\,|g_0(t)-g_{0,\,\varepsilon}(t)|\leq \varepsilon\quad \text{and}\quad \sup_{t\geq 0}\,|g_n(t)-g_{n,\,\varepsilon}(t)|\leq \varepsilon.
\end{equation}
We only prove the first inequality, the proof of the second being analogous and simpler. Write
\begin{eqnarray*}
|g_0(t)-g_{0,\,\varepsilon}(t)|&=&g_0(t)-g_{0,\,\varepsilon}(t)=\sup_{s\in
[0,t]}\,f_0(s)\vee \sup_{t_k\leq t}\,(f_0(t_k-)+y_k)\\&-&
\sup_{s\in [0,t]}\,f_0(s)\vee \sup_{\bar{t}_k\leq
t}\,(f_0(\bar{t}_k-)+\bar{y}_k)
\end{eqnarray*}
for all $t\geq 0$. There are two possibilities: either
$$\sup_{t_k\leq t}\,(f_0(t_k-)+y_k)=\sup_{t_k\leq t,\,t_k\neq
\bar{t}_k}\,(f_0(t_k-)+y_k)\vee \sup_{\bar{t}_k\leq
t}\,(f_0(\bar{t}_k-)+\bar{y}_k)=\sup_{\bar{t}_k\leq
t}\,(f_0(\bar{t}_k-)+\bar{y}_k)$$ in which case
$|g_0(t)-g_{0,\,\varepsilon}(t)|=0$ for all $t\geq 0$, i.e., the
first inequality in \eqref{121} holds, or
$$\sup_{t_k\leq t}\,(f_0(t_k-)+y_k)=\sup_{t_k\leq t,\,t_k\neq
\bar{t}_k}\,(f_0(t_k-)+y_k).$$ Observe that $$\sup_{t_k\leq
t,\,t_k\neq \bar{t}_k}\,(f_0(t_k-)+y_k)\leq \sup_{s\in
[0,t]}\,f_0(s)+ \varepsilon$$ as a consequence of $y_k\leq
\varepsilon$ for all $k\in\mn$ such that $t_k\neq \bar{t}_k$, and
that
$$\sup_{s\in [0,t]}\,f_0(s)\vee \sup_{\bar{t}_k\leq
t}\,(f_0(\bar{t}_k-)+\bar{y}_k)\geq \sup_{s\in [0,t]}\,f_0(s).$$
Hence, the first inequality in \eqref{121} holds in this case,
too.

It remains to note that
\begin{eqnarray*}
|g_n(T)-g_0(T)|&\leq&
|g_n(T)-g_{n,\varepsilon}(T)|+|g_{n,\varepsilon}(T)-g_{0,\varepsilon}(T)|+|g_{0,\varepsilon}(T)-g_0(T)|\\&\leq&
2\varepsilon+|g_{n,\varepsilon}(T)-g_{0,\varepsilon}(T)|
\end{eqnarray*}
and then first let $n$ tend to $\infty$ and use \eqref{aux}, and
then let $\varepsilon$ go to zero through the sequence
$(\varepsilon_k)$. This shows that $\lin g_n(T)=g_0(T)$. The proof of
the $M_1$-convergence is complete.

\noindent {\sc Proof of the $J_1$-convergence}. We intend to prove
that whenever $\bar{s}$ is a discontinuity point of
$g_{0,\,\varepsilon}$ there is a sequence $(s_n)_{n\in\mn}$
converging to $\bar{s}$ for which
\begin{equation}\label{127}
\lin g_{n,\varepsilon}(s_n)=g_{0,\varepsilon}(\bar{s})\quad\text{and}\quad \lin g_{n,\varepsilon}(s_n-)=g_{0,\varepsilon}(\bar{s}-).
\end{equation}

Now we explain that \eqref{127} entails
\begin{equation}\label{des}
\lin g_n=g_0
\end{equation}
in the $J_1$-topology on $D$, which is the desired result. From
the first part of the proof we know that $\lin
g_{n,\varepsilon}=g_{0,\varepsilon}$ in the $M_1$-topology on $D$.
Thus, if \eqref{127} holds, we conclude that
\begin{equation}\label{good}
\lin g_{n,\varepsilon}=g_{0,\varepsilon}
\end{equation}
in the $J_1$-topology on $D$ by Lemma \ref{lem1}(b). Let $r\in
[0,T]$ be a continuity point of $g_0$, where $T\in K$ (see the
first part of the proof for the definition of $K$). In order to
prove \eqref{des} it suffices to show that $\lin g_n=g_0$ in the
$J_1$-topology on $D[0,r]$ or, equivalently, that $\lin
\rho(g_n,g_0)=0$ where $\rho$ is the standard Skorokhod metric on
$[0,r]$. Since $r$ is also a continuity point of
$g_{0,\varepsilon}$, relation \eqref{good} ensures that $\lin
\rho(g_{n,\varepsilon}, g_{0,\varepsilon})=0$. We proceed by
writing
\begin{eqnarray*}
\rho(g_n,g_0)&\leq&
\rho(g_n,g_{n,\varepsilon})+\rho(g_{n,\varepsilon},
g_{0,\varepsilon})+\rho(g_{0, \varepsilon},g_0)\leq \sup_{0\leq
t\leq r}\,|g_n(t)-g_{n,\varepsilon}(t)|+\rho(g_{n,\varepsilon},
g_{0,\varepsilon})\\&+&\sup_{0\leq t\leq
r}\,|g_{0,\varepsilon}(t)-g_0(t)|\leq
2\varepsilon+\rho(g_{n,\varepsilon}, g_{0,\varepsilon})
\end{eqnarray*}
having utilized the fact that $\rho$ is dominated by the uniform
metric on $[0,r]$ for the penultimate inequality and \eqref{121} for the last. Now sending $n\to\infty$
and then letting $\varepsilon$ approach zero through the sequence
$(\varepsilon_k)$ proves $\lin \rho(g_n,g_0)=0$ and thereupon
\eqref{des}.

Passing to the proof of \eqref{127} we consider two cases.

\noindent {\sc Case 1}: $\bar{s}$ is a discontinuity point of
$g_{0,\varepsilon}$ and a continuity point of $f_0$.

We claim that $\bar{s}=\bar{t}_k$ for some $k=1,\ldots, m$.
Indeed, if $g_{0,\varepsilon}(\bar{s})=\sup_{t\in
[0,\bar{s}]}\,f_0(t)$, then $\bar{s}$ is a continuity point of
$g_{0,\varepsilon}$, a contradiction. Thus, we must have
$g_{0,\varepsilon}(\bar{s})=\max_{\bar{t}_j\leq
\bar{s}}\,(f_0(\bar{t}_j-)+\bar{y}_j)$. The points
$\bar{t}_1,\ldots$ $\bar{t}_m$ are the only discontinuities of
$x\mapsto \max_{\bar{t}_j\leq x}\,(f_0(\bar{t}_j-)+\bar{y}_j)$ on
$[0,\infty)$. Therefore, $\bar{s}=\bar{t}_k$ for some $k=1,\ldots,
m$, as claimed.

With this $k$, set $s_n=t_k^{(n)}-1/n$. Analogously to
\eqref{inter2} we obtain
\begin{equation}\label{01}
g_{n,\varepsilon}\big(\bar{t}_k^{(n)}-1/n\big)=\max_{t\in
[0,\bar{t}_k^{(n)}-2/n]}\, f_n(t)\vee \max_{1\leq j\leq
k}\,\big(f_n(\bar{t}_j^{(n)}-1/n)+\bar{y}_j^{(n)}\big)
\end{equation}
and
\begin{equation}\label{02}
g_{n,\varepsilon}\big((\bar{t}_k^{(n)}-1/n)-\big)=g_{n,\varepsilon}\big(\bar{t}_k^{(n)}-2/n\big)=\max_{t\in[0,\bar{t}_k^{(n)}-2/n]}\,
f_n(t)\vee \max_{1\leq j\leq
k-1}\,\big(f_n(\bar{t}_j^{(n)}-1/n)+\bar{y}_j^{(n)}\big).
\end{equation}
We shall now show
$$\lin g_{n,\varepsilon}(\bar{t}_k^{(n)}-1/n)=\sup_{t\in [0,\bar{t}_k]}\,f_0(t)\vee \max_{\bar{t}_j\leq
\bar{t}_k}\,\big(f_0(\bar{t}_j-)+\bar{y}_j\big)=g_{0,\varepsilon}(\bar{t}_k)$$
and
$$\lin g_{n,\varepsilon}((\bar{t}_k^{(n)}-1/n)-)=\sup_{t\in [0,\bar{t}_k]}\,f_0(t)\vee \max_{\bar{t}_j<\bar{t}_k}\,\big(f_0(\bar{t}_j-)+\bar{y}_j\big)=
g_{0,\varepsilon}(\bar{t}_k-).$$ Indeed, while the limit relations
\begin{equation}\label{con5}
\lin \max_{1\leq j\leq
k}\,\big(f_n(\bar{t}_j^{(n)}-1/n)+\bar{y}_j^{(n)}\big) =
\max_{\bar{t}_j\leq
\bar{t}_k}\,\big(f_0(\bar{t}_j-)+\bar{y}_j\big)
\end{equation}
and
\begin{equation}\label{con6}
\lin
\max_{1\leq j\leq
k-1}\,\big(f_n(\bar{t}_j^{(n)}-1/n)+\bar{y}_j^{(n)}\big)
=\max_{\bar{t}_j<\bar{t}_k}\,\big(f_0(\bar{t}_j-)+\bar{y}_j\big)
\end{equation}
are secured by \eqref{con1} and \eqref{con2}, the limit relation
$$\lin \max_{t\in [0,\bar{t}_k^{(n)}-2/n]}\, f_n(t)=\sup_{t\in [0,\bar{t}_k]}\,f_0(t)$$ holds in view of Lemma
\ref{imp} with $s_n=\bar{t}_k^{(n)}-1/n$ for $n\in\mn$ and
$s_0=\bar{t}_k$. Thus, formula \eqref{127} has been proved in Case
1.

\noindent {\sc Case 2}: $\bar{s}$ is a discontinuity point of both
$g_{0,\varepsilon}$ and $f_0$.

\noindent {\sc Subcase 2.1}: $\bar{s}=\bar{t}_k$ for some
$k=1,\ldots, m$. We intend to check that \eqref{127} holds with
$s_n=t_k^{(n)}-1/n$. Using formulae \eqref{01} and \eqref{02} and
recalling \eqref{con3}, \eqref{con5} and \eqref{con6} we infer
$$\lin g_{n,\varepsilon}((\bar{t}_k^{(n)}-1/n)-)=\sup_{t\in [0,\bar{t}_k)}\,f_0(t)\vee \max_{\bar{t}_j<\bar{t}_k}\,\big(f_0(\bar{t}_j-)+\bar{y}_j\big)=
g_{0,\varepsilon}(\bar{t}_k-)$$ and
$$\lin g_{n,\varepsilon}(\bar{t}_k^{(n)}-1/n)=\sup_{t\in [0,\bar{t}_k)}\,f_0(t)\vee \max_{\bar{t}_j\leq \bar{t}_k}\,\big(f_0(\bar{t}_j-)+\bar{y}_j\big).$$ Since
$$
f_0(\bar{t}_k)= f_0( \bar{t}_k-)+\bar{x}_k\leq f_0( \bar{t}_k-)+\bar y_k.
$$
in view of \eqref{eq1dop}, we conclude that $$\sup_{t\in [0,\bar{t}_k)}\,f_0(t)\vee \max_{\bar{t}_j\leq \bar{t}_k}\,\big(f_0(\bar{t}_j-)+\bar{y}_j\big)=
\sup_{t\in [0,\bar{t}_k]}\,f_0(t)\vee \max_{\bar{t}_j\leq \bar{t}_k}\,\big(f_0(\bar{t}_j-)+\bar{y}_j\big)=g_{0,\varepsilon}(\bar{t}_k),$$ thereby finishing the proof of \eqref{127} in this subcase.

\noindent {\sc Subcase 2.2}: $\bar{s}\notin \{\bar{t}_1,\ldots,
\bar{t}_m\}$. Let $r$ be a continuity point of $f_0$ satisfying
$r>\bar{s}$. Recall that \eqref{eq1} entails \eqref{supr}. Hence,
there is a sequence $(\lambda_n)_{n\in\mn}$ of continuous strictly
increasing functions of $[0,r]$ onto $[0,r]$ such that
$$
\lin \sup_{t\in[0,r]}|\lambda_n(t)-t|=0\quad \text{and}\quad \lin
\sup_{s\in[0,r]}|\sup_{t\in [0,\lambda_n(s)]}f_0(t)-\sup_{t\in
[0,s]}\,f_n(t)|=0.$$ In particular, $\lin
\sup_{t\in[0,s_n)}f_n(t)=\sup_{t\in[0,\bar{s})}f_0(t)$ and $\lin
\sup_{t\in[0,s_n]}f_n(t)=\sup_{t\in[0,\bar{s}]}f_0(t)$, where
$s_n:=\lambda_n(\bar{s})$. We shall show that \eqref{127} holds
with this choice of $s_n$. To this end, it only remains to note
that $$\lin \max_{\bar{t}_j^{(n)}<
{s}_n}\,\big(f_n(\bar{t}_j^{(n)})+\bar{y}_j^{(n)}\big)= \lin
\max_{\bar{t}_j^{(n)}\leq
s_n}\,\big(f_n(\bar{t}_j^{(n)})+\bar{y}_j^{(n)}\big)=
\max_{\bar{t}_j\leq
\bar{s}}\,\big(f_0(\bar{t}_j-)+\bar{y}_j\big)=\max_{\bar{t}_j<\bar{s}}\,\big(f_0(\bar{t}_j-)+\bar{y}_j\big)$$
as a consequence of $\bar{s}\notin \{\bar{t}_1,\ldots,
\bar{t}_m\}$ and \eqref{con1}. Therefore,
\begin{eqnarray*}
\lin g_{n,\varepsilon}(s_n-)&=&\lin \sup_{t\in
[0,s_n)}\,f_n(t)\vee
\max_{\bar{t}_j^{(n)}<s_n}\,\big(f_n(\bar{t}_j^{(n)})+\bar{y}_j^{(n)}\big)\\&=&
\sup_{t\in [0,\bar{s})}\,f_0(t)\vee
\max_{\bar{t}_j<\bar{s}}\,\big(f_0(\bar{t}_j
)+\bar{y}_j^{(n)}\big)=g_{0,\varepsilon}(\bar{s}-)
\end{eqnarray*}
and
\begin{eqnarray*}
\lin g_{n,\varepsilon}(s_{n})&=&\lin \sup_{t\in
[0,s_n]}\,f_n(t)\vee \max_{\bar{t}_j^{(n)}\leq
s_n}\,\big(f_n(\bar{t}_j^{(n)})+\bar{y}_j^{(n)}\big)\\&=&
\sup_{t\in [0,\bar{s}]}\,f_0(t)\vee
\max_{\bar{t}_j\leq\bar{s}}\,\big(f_0(\bar{t}_j)+\bar{y}_j^{(n)}\big)=g_{0,\varepsilon}(\bar{s})
\end{eqnarray*}
which proves \eqref{127}.

The proof of Theorem \ref{main10} is complete.
\end{proof}

\begin{proof}[Proof of Theorem \ref{main}]
By Corollary 6.1 on p.~183 in \cite{Resnick:2007}, condition
\eqref{princ} entails $$\sum_{l\geq 1}\varepsilon_{(l/n,\,
\xi_l/a(n),\,
\eta^+_l/a(n))}~\Rightarrow~\sum_k\varepsilon_{(\theta_k,i_k,j_k)},\quad
n\to\infty$$ on $\mathcal{M}_p([0,\infty)\times E)$ and thereupon
\begin{equation}\label{inter}
\bigg(\sum_{l\geq 1} \1_{\{ \xi_l\neq0 \}}\varepsilon_{(l/n,\,
\xi_l/a(n))},\, \sum_{l\geq 1}\varepsilon_{(l/n,\, \xi_l/a(n),\,
\eta^+_l/a(n))}\bigg)~\Rightarrow~\bigg(\sum_k \1_{\{ i_k\neq 0
\}}\varepsilon_{(\theta_k,i_k)},\,\sum_k\varepsilon_{(\theta_k,i_k,j_k)}\bigg)
\end{equation}
as $n\to\infty$ on $\mathcal{M}_p([0,\infty)\times
([-\infty,+\infty]\backslash\{0\}))\times
\mathcal{M}_p([0,\infty)\times E)$ because the first coordinates
are just the restrictions of the second from $[0,\infty)\times E$
on $[0,\infty)\times ([-\infty,+\infty]\backslash\{0\})$.

In the proof of Corollary 7.1 on p.~218 in \cite{Resnick:2007} it
is shown that the convergence of the first coordinates in
\eqref{inter} implies $S_{[n\cdot]}/a(n)\overset{{\rm
J_1}}{\Rightarrow}\mathcal{S}_\alpha(\cdot)$ as $n\to\infty$.
Starting with {\it full} relation \eqref{inter} exactly the same
reasoning leads to the conclusion
\begin{equation*}
\bigg(\frac{S_{[n\cdot]}}{a(n)},\, \sum_{l\geq
1}\varepsilon_{(l/n,\, \xi_l/a(n),\,
\eta^+_l/a(n))}\bigg)~\Rightarrow~\bigg(\mathcal{S}_\alpha^\ast(\cdot),\,\sum_k\varepsilon_{(\theta_k,i_k,j_k)}\bigg),\quad
n\to\infty
\end{equation*}
or, equivalently,
\begin{equation*}\label{inter1}
\bigg(\frac{S_{[n\cdot]}}{a(n)},\, \sum_{l\geq
1}\varepsilon_{(l/n,\, \xi_l/a(n),\,
\eta_l/a(n))}\1_{\{\eta_l>0\}}\bigg)~\Rightarrow~\bigg(\mathcal{S}_\alpha^\ast(\cdot),\,\sum_k\varepsilon_{(\theta_k,i_k,j_k)}\bigg),\quad
n\to\infty
\end{equation*}
in the product topology on $D\times \mathcal{M}_p([0,\infty)\times
E)$. By the Skorokhod representation theorem there are versions
which converge a.s. Retaining the original notation for these
versions we want to apply Theorem \ref{main10} with
$f_n(\cdot)=S_{[n\cdot]}/a(n)$, $f_0=\mathcal{S}_\alpha^\ast$,
$\nu_n=\sum_{l\geq 1}\varepsilon_{\{l/n,\,\xi_l/a(n),\,
\eta_l/a(n)\}}\1_{\{\eta_l>0\}}$ and $\nu_0=
N^{(\nu)}=\sum_k\varepsilon_{(\theta_k,i_k,j_k)}$. We already know
that conditions \eqref{eq1} and \eqref{eq1_1} are fulfilled a.s.
It is obvious that $N^{(\nu)}(\{0\}\times E)=0$ a.s. In order to
show that $N^{(\nu)}$ does not have clustered jumps a.s.\, i.e.,
$\theta_k\neq \theta_j$ for $k\neq j$ a.s., it suffices to check
this property for $N^{(\nu)}(([0,T]\times [-\infty,+\infty]\times
(\delta,\infty])\cap \cdot)$ with $T>0$ and $\varepsilon>0$ fixed.
This is done on p.~223 in \cite{Resnick:2007}. Hence Theorem
\ref{main10} is indeed applicable with our choice of $f_n$ and
$\nu_n$, and \eqref{limit_perturbed01} follows.
\end{proof}

\section{Proofs of Propositions \ref{main1} and \ref{main2} and Theorem \ref{main3}}

\begin{proof}[Proof of Proposition \ref{main1}]
Fix any $T>0$. Note that \eqref{cond1} entails
\begin{equation}\label{cond11}
\lix x\mmp\{\eta>\varepsilon a(x)\}=0
\end{equation}
for all $\varepsilon>0$ because $a(x)$ is regularly varying at
$\infty$ (of index $1/\alpha$). Since, for all $\varepsilon>0$,
\begin{eqnarray*}
\mmp\{\sup_{0\leq s\leq T}\,\eta_{[ns]+1}>\varepsilon
a(n)\}&=&1-\big(\mmp\{\eta\leq\varepsilon
a(n)\}\big)^{[nT]+1}\\&\leq& ([nT]+1)\mmp\{\eta>\varepsilon
a(n)\}\to 0
\end{eqnarray*}
as $n\to\infty$ in view of \eqref{cond11}, we infer
$$\frac{\sup_{0\leq s\leq T}\,\eta_{[ns]+1}}{a(n)}~\tp~ 0, \quad
n\to\infty.
$$ This in combination with \eqref{skor} enables us to conclude that
$$\frac{S_{[n\cdot]}+\eta_{[n\cdot]+1}}{a(n)}~\overset{{\rm J_1}}{\Rightarrow}~\mathcal{S}_\alpha(\cdot),\quad n\to\infty$$ by Slutsky's lemma.
Relation \eqref{limit_perturbed0} now follows by the continuous
mapping theorem because the supremum functional is continuous in
the $J_1$-topology.
\end{proof}

\begin{proof}[Proof of Proposition \ref{main2}]
To begin with, we note that $\lin (b(n)/a(n))=\infty$ as a
consequence of \eqref{cond12}. Consequently,
$$S_{[n\cdot]}/b(n)~\overset{{\rm J_1}}{\Rightarrow}~ \Xi(\cdot)$$ in view of \eqref{skor}, where $\Xi(t):=0$ for $t\geq 0$. Further, according to Theorem 3.6
on p.~62 in combination with Corollary 6.1 on p.~183 in
\cite{Resnick:2007}, regular variation of $\mmp\{\eta>x\}$ ensures
that
$$\sum_{k\geq 0}\varepsilon_{(k/n,
\eta_{k+1}/b(n))}\1_{\{\eta_{k+1}>0\}}~\Rightarrow~
N:=\sum_{k}\varepsilon_{(\theta_k, j_k)},\quad n\to\infty$$ on
$\mathcal{M}_p([0,\infty)\times (0,\infty])$ and thereupon
$$\bigg(S_{[n\cdot]}/b(n), \sum_{k\geq
0}\1_{\{\eta_{k+1}>0\}}\varepsilon_{(k/n,
\eta_{k+1}/b(n))}\bigg)~\Rightarrow~ \big(\Xi(\cdot), N\big),\quad
n\to\infty$$ on $D\times \mathcal{M}_p([0,\infty)\times
(0,\infty])$ equipped with the product topology. Arguing as in the
proof of Theorem \ref{main} we obtain \eqref{limit_perturbed1} by
an application of Remark \ref{re2} with
$f_n(\cdot)=S_{[n\cdot]}/b(n)$, $f_0=\Xi$, $\nu_n=\sum_{k\geq
0}\varepsilon_{(k/n, \eta_{k+1}/b(n))}\1_{\{\eta_{k+1}>0\}}$ and
$\nu_0=N$. The condition $N((a,b)\times (0,\infty])\geq 1$ a.s.\
whenever $0<a<b$ required in Remark \ref{re2} holds because
$\mu((0,\infty])=\infty$.
\end{proof}

\begin{proof}[Proof of Theorem \ref{main3}]
The limit relations of Theorem \ref{main} and Propositions
\ref{main1} and \ref{main2} can be written in a unified form as
$$\frac{\max_{0\leq k\leq [n\cdot]}\,T_{k+1}}{c(n)}~\Rightarrow~ X(\cdot),\quad
n\to\infty$$ in the $J_1$- or the $M_1$-topology on $D$. Using
this limit relation together with the inequality
\begin{eqnarray*} \max_{0\leq k\leq
n}\,T_{k+1}\leq \log\bigg(\sum_{k=0}^n e^{T_{k+1}}\bigg)\leq \log
(n+1)+ \max_{0\leq k\leq n}\,T_{k+1}
\end{eqnarray*}
and the fact that $\lin (\log n/c(n))=0$ we arrive at the desired
conclusion
$$\frac{\log\sum_{k=0}^{[n\cdot]}e^{T_{k+1}}}{c(n)}~\Rightarrow~ X(\cdot),\quad
n\to\infty$$ in the $J_1$- or $M_1$-topology on $D$.

Since the limit process $X$ is nonnegative a.s., the result
remains true on replacing $\log$ with $\log ^+$.
\end{proof}


\end{document}